\documentclass[12 pt,twoside]{amsart}

\usepackage{amsmath,amsfonts,amsthm,amssymb}

\newtheorem{theorem}{Theorem}[section]
\newtheorem{lemma}[theorem]{Lemma}
\newtheorem{proposition}[theorem]{Proposition}

\theoremstyle{definition}

\theoremstyle{remark}
\newtheorem{remark}[theorem]{Remark}

\numberwithin{equation}{section}

\begin{document}

\title [Matrices forming locally hypercyclic tuples]{On the minimal number of matrices
which form a locally hypercyclic, non-hypercyclic tuple}

\author[G. Costakis]{G. Costakis}
\address{Department of Mathematics, University of Crete, Knossos Avenue, GR-714 09 Heraklion, Crete, Greece}
\email{costakis@math.uoc.gr}
\thanks{}

\author[D. Hadjiloucas]{D. Hadjiloucas}
\address{Department of Computer Science and Engineering, The School of Sciences, European University Cyprus, 6
Diogenes Street, Engomi, P.O.Box 22006, 1516 Nicosia, Cyprus}
\email{d.hadjiloucas@euc.ac.cy}
\thanks{}

\author[A. Manoussos]{A. Manoussos}
\address{Fakult\"{a}t f\"{u}r Mathematik, SFB 701, Universit\"{a}t Bielefeld, Postfach 100131, D-33501 Bielefeld, Germany}
\email{amanouss@math.uni-bielefeld.de}
\thanks{During this research the third author was fully supported by SFB 701 ``Spektrale Strukturen und
Topologische Methoden in der Mathematik" at the University of Bielefeld, Germany. He would also like to express
his gratitude to Professor H. Abels for his support.}

\date{}



\begin{abstract}
In this paper we extend the notion of a locally hypercyclic operator to that of a locally hypercyclic tuple of
operators. We then show that the class of hypercyclic tuples of operators forms a proper subclass to that of
locally hypercyclic tuples of operators. What is rather remarkable is that in every finite dimensional vector
space over $\mathbb{R}$ or $\mathbb{C}$, a pair of commuting matrices exists which forms a locally hypercyclic,
non-hypercyclic tuple. This comes in direct contrast to the case of hypercyclic tuples where the minimal number
of matrices required for hypercyclicity is related to the dimension of the vector space. In this direction we
prove that the minimal number of diagonal matrices required to form a hypercyclic tuple on $\mathbb{R}^n$ is
$n+1$, thus complementing a recent result due to Feldman.
\end{abstract}

\maketitle

\section{Introduction}
Locally hypercyclic (or $J$-class) operators form a class of linear operators which possess certain dynamic
properties. These were introduced and studied in \cite{cosma1}. The notion of a locally hypercyclic operator can
be viewed as a ``localization" of the notion of hypercyclic operator. For a comprehensive study and account of
results on hypercyclic operators we refer to the book \cite{BaMa} by Bayart and Matheron.

Hypercyclic tuples of operators were
introduced and studied by Feldman in \cite{Feldman1},
\cite{Feldman2} and \cite{Feldman3}, see also \cite{Kerchy}.
An $n$-tuple of operators is a
finite sequence of length $n$ of commuting continuous linear
operators $T_1, T_2,\ldots,T_n$ acting on a locally convex
topological vector space $X$. The tuple $(T_1,T_2,\ldots,T_n)$ is
hypercyclic if there exists a vector $x\in X$ such that the set
\[\left\{T_1^{k_1}T_2^{k_2}\ldots T_n^{k_n}x :
k_1, k_2,\ldots, k_n \in\mathbb{N}\cup\{0\}\right\}\] is dense in $X$. The tuple $(T_1,T_2,\ldots,T_n)$ is topologically transitive if for every pair $(U,V)$ of
non-empty open sets in $X$ there exist $k_1, k_2, \ldots ,k_n$ $\in\mathbb{N}\cup\{0\}$ such that $T_1^{k_1}T_2^{k_2}\ldots T_n^{k_n}(U)\cap V\neq \emptyset$. If $X$
is separable it is easy to show that $(T_1,T_2,\ldots,T_n)$ is topologically transitive if and only if $(T_1,T_2,\ldots,T_n)$ is hypercyclic. Following Feldman
\cite{Feldman3}, we denote the semigroup generated by the tuple $T=(T_1,T_2,\ldots, T_n)$ by $\mathcal{F}_T=\{T_1^{k_1}T_2^{k_2}\ldots T_n^{k_n}: k_i
\in\mathbb{N}\cup\{0\}\}$ and the orbit of $x$ under the tuple $T$ by $Orb(T,x)=\{Sx:S\in\mathcal{F_T}\}$. Furthermore, we denote by $HC((T_1, T_2,\ldots,T_n))$ the
set of hypercyclic vectors for the tuple $(T_1,T_2,\ldots,T_n)$.

In this article we extend the notion of a locally hypercyclic operator (locally topologically transitive) to
that of a locally hypercyclic tuple (locally topologically transitive tuple) of operators as follows. For $x\in
X$ we define the extended limit set $J_{(T_1,T_2,\ldots,T_n)}(x)$ to be the set of $y\in X$ for which there
exist a sequence of vectors $\{x_m\}$ with $x_m\to x$ and sequences of non-negative integers
$\{k_m^{(j)}:m\in\mathbb{N}\}$ for $j=1,2,\ldots,n$ with
\begin{eqnarray}\label{unbseq}
k_m^{(1)}+k_m^{(2)}+\ldots+k_m^{(n)}&\to& +\infty
\end{eqnarray}
such that
\[T_1^{k_m^{(1)}}T_2^{k_m^{(2)}}\ldots T_n^{k_m^{(n)}} x_m\to y.\]
Note that condition (\ref{unbseq}) is equivalent to having at least one of the sequences $\{k_m^{(j)}:m\in\mathbb{N}\}$ for $j=1,2,\ldots,n$ containing a strictly
increasing subsequence tending to $+\infty$. This is in accordance with the well-known definition of $J$-sets in topological dynamics, see \cite{Hajek}. In section
$2$ we provide an explanation as to why condition (\ref{unbseq}) is reasonable. The tuple $(T_1,T_2,\ldots,T_n)$ is \textit{locally topologically transitive} if there
exists $x\in X\setminus\{0\}$ such that $J_{(T_1,T_2,\ldots,T_n)}(x)=X$. Using simple arguments it is easy to show the following equivalence.
$J_{(T_1,T_2,\ldots,T_n)}(x)=X$ if and only if for every open neighborhood $U_x$ of $x$ and every non-empty open set $V$ there exist $k_1, k_2, \ldots
,k_n\in\mathbb{N}\cup\{0\}$ such that $T_1^{k_1}T_2^{k_2}\ldots T_n^{k_n}(U_x)\cap V\neq \emptyset$. In the case $X$ is separable and there exists $x\in
X\setminus\{0\}$ such that $J_{(T_1,T_2,\ldots,T_n)}(x)=X$, the tuple $(T_1,T_2,\ldots,T_n)$ will be called
\textit{locally hypercyclic}.

In a finite dimensional space over $\mathbb{R}$ or $\mathbb{C}$, no linear operator can be hypercyclic (see
\cite {Kitai}) or locally hypercyclic (see \cite{cosma1}). However, it was shown recently by Feldman in
\cite{Feldman3} that the situation for tuples of linear operators in finite dimensional spaces over $\mathbb{R}$
or $\mathbb{C}$ is quite different. There, it was shown that there exist hypercyclic $(n+1)$-tuples of diagonal
matrices on $\mathbb{C}^n$ and that no $n$-tuple of diagonal matrices is hypercyclic. We complement this result
by showing that the minimal number of diagonal matrices required to form a hypercyclic tuple in $\mathbb{R}^n$
is $n+1$. We also mention at this point that in \cite{cohama1} it is proved that non-diagonal hypercyclic
$n$-tuples exist on $\mathbb{R}^n$, answering a question of Feldman.

In the present work we make a first attempt towards studying locally hypercyclic tuples of linear operators on
finite dimensional vector spaces over $\mathbb{R}$ or $\mathbb{C}$. We show that if a tuple of linear operators
is hypercyclic then it is locally hypercyclic (see section $2$). We then proceed to show that in the finite
dimensional setting, the class of hypercyclic tuples of operators forms a proper subclass of the class of
locally hypercyclic tuples of operators. What is rather surprising is the fact that the minimal number of
matrices required to construct a locally hypercyclic tuple in any finite dimensional space over $\mathbb{R}$ or
$\mathbb{C}$ is $2$. This comes in direct contrast to the class of hypercyclic tuples where the minimal number
of matrices required depends on the dimension of the vector space. Examples of diagonal pairs of matrices as
well as pairs of upper-triangular non-diagonal matrices and matrices in Jordan form which are locally
hypercyclic but not hypercyclic are constructed. We mention that some of our constructions can be directly
generalized to the infinite dimensional case, see section $4$.

\section{Basic properties of locally hypercyclic tuples of operators}
Let us first comment on the condition (\ref{unbseq}) in the definition of a locally hypercyclic tuple. This
comes as an extension to the definition of a locally hypercyclic operator given in \cite{cosma1}. Recall that a
hypercyclic operator $T:X\to X$ is locally hypercyclic and furthermore $J_T(x)=X$ for every $x\in X$. In the
definition of a locally hypercyclic tuple, one may have been inclined to demand that $k_m^{(j)}\to +\infty$ for
every $j=1,2,\ldots,n$. However this would lead to a situation where the class of hypercyclic tuples would not
form a subclass of the locally hypercyclic tuples. To clarify this issue, we give an example. Take any
hypercyclic operator $T:X\to X$ and consider the tuple $(T,0)$ where $0:X\to X$ is the zero operator defined by
$0(x)=0$ for every $x\in X$. Obviously, this is a hypercyclic tuple ($Orb(0,x)=\{x,0\}$). On the other hand, for
every pair of sequences of integers $\{n_k\}$, $\{m_k\}$ with $n_k, m_k\to +\infty$ and for every sequence of
vectors $x_k$ tending to some vector $x$ we have $T^{n_k}0^{m_k}x_k\to 0$ and so $(T,0)$ would not be a locally
hypercyclic pair.

Let us now proceed by stating some basic facts which will be used in showing that the class of hypercyclic
tuples is contained in the class of locally hypercyclic tuples.

\begin{lemma}\label{basiclemma1}
If $x\in HC((T_1,T_2,\ldots,T_n))$ then
$J_{(T_1,T_2,\ldots,T_n)}(x)=X$.
\end{lemma}

\begin{proof}
Let $y\in X$, $\epsilon>0$ and $m\in\mathbb{N}$. Since the set
$Orb(T,x)$ is dense in $X$ it follows that the set
\[\{T_1^{k_1}T_2^{k_2}\ldots T_n^{k_n} x: k_1+k_2+\ldots+k_n>m\}\]
is dense in $X$ (only a finite number of vectors are omitted from the orbit $Orb(T,x)$). Hence, there
exist $(k_1,k_2,\ldots,k_n)\in\mathbb{N}^n$ with
$k_1+k_2+\ldots+k_n>m$ such that
\[\|T_1^{k_1}T_2^{k_2}\ldots T_n^{k_n} x - y\|<\epsilon.\]
\end{proof}
The proof of the following lemma is an immediate variation of the
proof of lemma 2.5 in \cite{cosma2}.
\begin{lemma}\label{basiclemma2}
If $\{x_m\}$, $\{y_m\}$ are two sequences in $X$ such that $x_m\to
x$ and $y_n\to y$ for some $x,y\in X$ and $y_m\in
J_{(T_1,T_2,\ldots,T_n)}(x_m)$ for every $m\in\mathbb{N}$ then $y\in
J_{(T_1,T_2,\ldots,T_n)}(x)$.
\end{lemma}
\begin{lemma}\label{basiclemma3}
For all $x\in X$ the set $J_{(T_1,T_2,\ldots,T_n)}(x)$ is closed and
$T_j$ invariant for every $j=1,2,\ldots,n$.
\end{lemma}
\begin{proof}
This is an easy consequence of Lemma \ref{basiclemma2}.
\end{proof}
\begin{proposition}
$(T_1,T_2,\ldots,T_n)$ is hypercyclic if and only if it is locally hypercyclic and
$J_{(T_1,T_2,\ldots,T_n)}(x)=X$ for every $x\in X$.
\end{proposition}
\begin{proof}
Assume first that $(T_1,T_2,\ldots,T_n)$ is hypercyclic. By Lemma \ref{basiclemma1} it follows that
$(T_1,T_2,\ldots,T_n)$ is locally hypercyclic. Denote by $A$ the set of vectors $\{x\in X:
J_{(T_1,T_2,\cdots,T_n)}(x)=X\}$. By Lemma \ref{basiclemma1} we have $HC((T_1,T_2,\ldots,T_n))\subset A$. Since
$HC((T_1,T_2,\ldots,T_n))$ is dense (see \cite{Feldman3}) and $A$ is closed by Lemma \ref{basiclemma2}, it is
plain that $A=X$. For the converse implication let us consider $x\in X$. Since $J_{(T_1,T_2,\ldots,T_n)}(x)=X$
then for every open neighborhood $U_x$ of $x$ and every non-empty open set $V$ there exist $k_1, k_2, \ldots
k_n\in\mathbb{N}\cup\{0\}$ such that $T_1^{k_1}T_2^{k_2}\ldots T_n^{k_n}(U_x)\cap V\neq \emptyset$. Therefore
$(T_1,T_2,\ldots,T_n)$ is topologically transitive and since $X$ is separable it follows that
$(T_1,T_2,\ldots,T_n)$ is hypercyclic.
\end{proof}

\section{Locally hypercyclic pairs of diagonal matrices which are not hypercyclic}
In \cite{Feldman3}, Feldman showed that there exist $(n+1)$-tuples
of diagonal matrices on $\mathbb{C}^n$ and that there are no
hypercyclic $n$-tuples of diagonalizable matrices on $\mathbb{C}^n$.
In the same paper, Feldman went a step further to show that no $n$-tuple
of diagonal matrices on $\mathbb{R}^n$ is hypercyclic while, on the other
hand, there exists an $(n+1)$-tuple of diagonal matrices on $\mathbb{R}^n$
that has a dense orbit in $(\mathbb{R}^+)^n$.
We complement the last result by showing that there is an $(n+1)$-tuple
of diagonal matrices on $\mathbb{R}^n$ which is hypercyclic.
Throughout the rest of the paper for a vector $u$ in $\mathbb{R}^n$ or $\mathbb{C}^n$
we will be denoting by
$u^t$ the transpose of $u$.
\begin{theorem}
For every $n\in\mathbb{N}$ there exists an $(n+1)$-tuple of diagonal matrices
on $\mathbb{R}^n$ which is hypercyclic.
\end{theorem}
\begin{proof}
Choose negative real numbers $a_1, a_2,\ldots,a_n$ such that the numbers
\[1, a_1, a_2,\ldots, a_n\]
are linearly independent over $\mathbb{Q}$. By Kronecker's theorem (see theorem 442 in \cite{HardyWright}) the
set
\[\{(k a_1 +s_1, k a_2+s_2,\ldots, k a_n+s_n)^t: k, s_1,\ldots, s_n\in\mathbb{N}\cup\{0\}\}\]
is dense in $\mathbb{R}^n$. The continuity of the map $f:\mathbb{R}^n\to\mathbb{R}^n$
defined by $f(x_1,x_2,\ldots,x_n)=(e^{x_1},e^{x_2},\ldots,e^{x_n})$ implies that
the set
\[\{((e^{a_1})^k e^{s_1}, (e^{a_2})^k e^{s_2},\ldots, (e^{a_n})^k e^{s_n})^t:k,s_1,\ldots,s_n\in
\mathbb{N}\cup\{0\}\}\]
is dense in $(\mathbb{R}^+)^n$. An easy argument (see for example the proof of lemma 2.6 in \cite{cohama1})
shows that the set
\[\left\{\left(
\begin{array}{c}
(e^{a_1})^k (-\sqrt{e})^{s_1}\\
(e^{a_2})^k (-\sqrt{e})^{s_2}\\
\vdots\\
(e^{a_n})^k (-\sqrt{e})^{s_n}
\end{array} \right):k,s_1,\ldots,s_n\in\mathbb{N}\cup\{0\}\right\}\]
is dense in $\mathbb{R}^n$.
Let
\[\textbf{1}=\left(
\begin{array}{c}
1\\
1\\
\vdots\\
1
\end{array}\right),\quad
A=
\left(
\begin{array}{cccc}
e^{a_1} & & & \\
 & e^{a_2} & & \\
 & & \ddots & \\
 & & & e^{a_n}
\end{array}
\right)
\]
\[B_1= \left(
\begin{array}{cccc}
-\sqrt{e} & & & \\
 & 1 & & \\
 & & \ddots & \\
 & & & 1
\end{array}\right),\; \ldots\;,
B_n=\left(\begin{array}{cccc}
1 & & & \\
 & 1 & & \\
 & & \ddots & \\
 & & & -\sqrt{e}
\end{array}\right).
\]
Then the set
\[\{A^k B_1^{s_1}\ldots B_n^{s_n} \textbf{1}:k,s_1,\ldots,s_n\in\mathbb{N}\cup\{0\}\}\]
is dense in $\mathbb{R}^n$, which implies that the $(n+1)$-tuple
$(A,B_1,\ldots,B_n)$ of diagonal matrices is hypercyclic.
\end{proof}
All of the results mentioned at the beginning of this section as well as the one proved above show that the
length of a hypercyclic tuple of diagonal matrices depends on the dimension of the space. It comes as a surprise
that this is not the case for locally hypercyclic tuples of diagonal matrices. In fact, we show that on a vector
space of any finite dimension $n\geq 2$ one may construct a pair of diagonal matrices which is locally
hypercyclic.
\begin{theorem}\label{diagonal}
Let $a,b \in\mathbb{R}$ such that $-1<a<0$, $b>1$ and $\frac{ \ln |a|}{\ln b}$ is irrational. Let $n$ be a
positive integer with $n\geq 2$ and consider the $n\times n$ matrices
\[A=
\left(
\begin{array}{lllcl}
a_1 & 0 & 0 & \ldots & 0\\
0 & a_2 & 0 & \ldots & 0\\
0 & 0 & a_3 & \ldots & 0\\
\vdots & \vdots & \vdots & \ddots & 0\\
0& 0& 0& \ldots & a_n
\end{array}
\right),\quad B= \left(
\begin{array}{lllcl}
b_1 & 0 & 0 & \ldots & 0\\
0 & b_2 & 0 & \ldots & 0\\
0 & 0 & b_3 & \ldots & 0\\
\vdots & \vdots & \vdots & \ddots & 0\\
0& 0& 0& \ldots & b_n
\end{array}\right)
\]
where $a_1=a$, $b_1=b$, $a_j, b_j$ real numbers with $|a_j|>1$ and $|b_j|>1$ for $j=2,\ldots n$. Then $( A, B )$
is a locally hypercyclic pair on $\mathbb{R}^n$ which is not hypercyclic. In particular, we have
\[\{x\in\mathbb{R}^n:J_{(A,B)}(x)=\mathbb{R}^n\}=\{(x_1,0,\ldots,0)^t\in\mathbb{R}^n: x_1\in\mathbb{R}\}.\]
\end{theorem}
\begin{proof}
Note that
\[
A^k B^l= \left(
\begin{array}{lllcl}
a^k b^l & 0       & 0       & \ldots & 0\\
0      & a_2^k b_2^l & 0       & \ldots & 0\\
0      & 0       & a_3^k b_3^l & \ldots & 0\\
\vdots & \vdots  & \vdots  & \ddots & 0\\
0      & 0      & 0        & \ldots & a_n^k b_n^l
\end{array}
\right).
\]
Let $x=(1,0,0,\ldots,0)^t\in\mathbb{R}^n$. We will show that $J_{(A,B)}(x)=\mathbb{R}^n$. Fix a vector
$y=(y_1,\ldots,y_n)^t$. By \cite[lemma 2.6]{cohama1}, the sequence $\{ a^k b^l : k,l\in \mathbb{N} \}$ is dense
in $\mathbb{R}$. Hence there exist sequences of positive integers $\{k_i\}$ and $\{l_i\}$ with $k_i,l_i\to
+\infty$ such that $a^{k_i}b^{l_i}\to y_1$. Let \[x_i=\left(1,\frac{y_2}{a_2^{k_i} b_2
^{l_i}},\ldots,\frac{y_n}{a_n^{k_i}b_n^{l_i}}\right)^t.\] Obviously $x_i\to x$ and
\[
A^{k_i}B^{l_i}x_i=(a^{k_i}b^{l_i},y_2,\ldots,y_n)^t\to y.
\]
In \cite[theorems 3.4 and 3.6]{Feldman3} Feldman showed that there
exists a hypercyclic $(n+1)$-tuple of diagonal matrices on
$\mathbb{C}^n$, for every $n\in\mathbb{N}$ but there is no
hypercyclic $n$-tuple of diagonal matrices on $\mathbb{C}^n$ or on
$\mathbb{R}^n$. Feldman actually showed that there is no $n$-tuple
of diagonal matrices on $\mathbb{C}^n$ or $\mathbb{R}^n$  that has a
somewhere dense orbit \cite[theorem 4.4]{Feldman3}. So the pair
$(A,B)$ is not hypercyclic. To finish, note that for every
$\lambda\in\mathbb{R}\setminus\{0\}$ it holds that $J_{(A,B)}(\lambda x)=
\lambda J_{(A,B)}(x)=\mathbb{R}^n$. In view of Lemma \ref{basiclemma2}
it follows that $J_{(A,B)}(0)=\mathbb{R}^n$. On the other hand, by
the choice of $a_j,b_j$ for $j=2,\ldots,n$ it is clear that for any
vector $u=(u_1,u_2,\ldots,u_n)^t$ with $u_j\neq 0$ for some $j\in\{2,3,\ldots
,n\}$ we have $J_{(A,B)}(u)\neq\mathbb{R}^n$. This completes the proof of the
theorem.
\end{proof}
A direct analogue to the previous theorem also holds in the complex
setting. We will make use of the following result in \cite{Feldman3}
due to Feldman.

\begin{proposition}
(i) If $b\in\mathbb{C}\setminus\{0\}$ with $|b|<1$ then there is a
dense set $\Delta_b\subset\{z\in\mathbb{C}:|z|>1\}$ such that for
any $a\in\Delta_b$, we have $\{a^k b^l:k,l\in\mathbb{N}\}$ is dense
in $\mathbb{C}$.

(ii) If $a\in\mathbb{C}$ with $|a|>1$, then there is a dense set
$\Delta_a\subset\{z\in\mathbb{C}:|z|<1\}$ such that for any
$b\in\Delta_a$, we have $\{a^k b^l: k,l\in\mathbb{N}\}$ is dense in
$\mathbb{C}$.

\end{proposition}

\begin{theorem}
Let $a,b\in\mathbb{C}$ such that $\{a^k b^l: k,l\in\mathbb{N}\}$ is dense in $\mathbb{C}$. Let $n$ be a positive
integer with $n\geq 2$ and consider the diagonal matrices $A$ and $B$ as in Theorem \ref{diagonal} where
$a_1=a$, $b_1=b$, $a_j, b_j\in\mathbb{C}$ with $|a_j|>1$ and $|b_j|>1$ for $j=2,\ldots,n$. Then $(A,B)$ is a
locally hypercyclic pair on $\mathbb{C}^n$ which is not hypercyclic. In particular, we have
\[\{z\in\mathbb{C}^n:J_{(A,B)}(z)=\mathbb{C}^n\}=\{(z_1,0,\ldots,0)^t\in\mathbb{C}^n: z_1\in\mathbb{C}\}.\]
\end{theorem}
\begin{proof}
The proof follows along the same lines as that of Theorem
\ref{diagonal}.
\end{proof}

\section{Locally hypercyclic pairs of diagonal operators which are not hypercyclic in infinite dimensional spaces}
In this section we slightly modify the construction in Theorem 3.2 in order to obtain similar results in
infinite dimensional spaces. As usual the symbol $l^p(\mathbb{N})$ stands for the Banach space of $p$-summable
sequences, where $1\leq p<\infty$ and by $l^{\infty}(\mathbb{N})$ we denote the Banach space of bounded
sequences (either over $\mathbb{R}$ or $\mathbb{C}$).

\begin{theorem}
Let $a,b\in\mathbb{C}$ such that $\{a^k b^l: k,l\in\mathbb{N}\}$ is dense in $\mathbb{C}$ and let $c\in
\mathbb{C}$ with $|c|>1$. Consider the diagonal operators $T_j:l^p(\mathbb{N})\to l^p(\mathbb{N})$, $1\leq p\leq
\infty$, $j=1,2$, defined by
$$T_1(x_1, x_2, x_3, \ldots )=(ax_1, cx_2, cx_3, \ldots ),$$
$$T_2(x_1, x_2, x_3, \ldots )=(bx_1, cx_2, cx_3, \ldots ),$$
for $x=(x_1, x_2, x_3, \ldots )\in l^p(\mathbb{N})$, $1\leq p\leq \infty$. Then $(T_1, T_2)$ is a locally
hypercyclic, non-hypercyclic pair in $l^p(\mathbb{N})$ for every $1\leq p<\infty$ and $(T_1, T_2)$ is a locally
topologically transitive, non-topologically transitive pair in $l^{\infty}(\mathbb{N})$. In particular we have
$$ \{ x\in l^p(\mathbb{N}): J_{(T_1,T_2)}(x)=l^p(\mathbb{N}) \}= \{ (x_1, 0,0,\ldots ): x_1 \in \mathbb{C} \}$$
for every $1\leq p\leq \infty$.
\end{theorem}
\begin{proof}
Fix $1\leq p\leq \infty$ and consider a vector $y=(y_1,y_2,\ldots )\in l^p(\mathbb{N})$. There exist sequences
of positive integers $\{k_i\}$ and $\{l_i\}$ with $k_i,l_i\to +\infty$ such that $a^{k_i}b^{l_i}\to y_1$. Let
\[x_i=\left(1,\frac{y_2}{c^{k_i +l_i}},\frac{y_3}{c^{k_i +l_i}},\ldots  \right).\] Obviously
$x_i\to x=(1,0,0,\ldots )$ and
\[
T_1^{k_i}T_2^{l_i}x_i=(a^{k_i}b^{l_i},y_2,y_3,\dots ) \to y.
\]
Therefore $J_{(T_1,T_2)}(x)=l^p(\mathbb{N})$. For $p=2$ the pair $(T_1,T_2)$ is not hypercyclic by Feldman's
result which says that there are no hypercyclic tuples of normal operators in infinite dimensions, see
\cite{Feldman3}. However, one can show directly that for every $1\leq p< \infty$ the pair $(T_1,T_2)$ is not
hypercyclic and $(T_1,T_2)$ is not topologically transitive in $l^{\infty}(\mathbb{N})$. Indeed, suppose that
$x=(x_1,x_2,\ldots )\in l^p(\mathbb{N})$ is hypercyclic for the pair $(T_1,T_2)$, where $1\leq p<\infty$. Then
necessarily $x_2\neq 0$ and the sequence $\{ c^n\}$ should be dense in $\mathbb{C}$ which is a contradiction.
For the case $p=\infty$, assuming that the pair $(T_1,T_2)$ is topologically transitive we conclude that the
pair $(A,B)$ is topologically transitive in $\mathbb{C}^2$, where $A(x_1,x_2)=(ax_1,cx_2)$,
$B(x_1,x_2)=(bx_1,cx_2)$, $(x_1,x_2)\in \mathbb{C}^2$. The latter implies that $(A,B)$ is hypercyclic. Since no
pair of diagonal matrices is hypercyclic in $\mathbb{C}^2$, see \cite{Feldman3}, we arrive at a contradiction.
It is also easy to check that $ \{ x\in l^p(\mathbb{N}): J_{(T_1,T_2)}(x)=l^p(\mathbb{N}) \}= \{ (x_1,
0,0,\ldots ): x_1 \in \mathbb{C} \}$ for every $1\leq p\leq \infty$.
\end{proof}

\begin{remark}
Theorem 4.1 is valid for the $l^p(\mathbb{N})$ spaces over the reals as well. Concerning the non-separable
Banach space $l^{\infty}(\mathbb{N})$ we stress that this space does not support topologically transitive
operators, see \cite{BeKa}. On the other hand there exist operators acting on $l^{\infty}(\mathbb{N})$ which are
locally topologically transitive, see \cite{cosma1}.
\end{remark}

\section{Locally hypercyclic pairs of upper triangular non-diagonal matrices which are not hypercyclic}
We first show that it is possible for numbers $a_1,a_2\in\mathbb{R}$ to exist with the property that the set
\[
\left\{\frac{a_1^{k} a_2^{l}}{\frac{k}{a_1}+\frac{l}{a_2}} : k,l\in\mathbb{N}\right\}
\]
is dense in $\mathbb{R}$ and at the same time the sequences on both the numerator and denominator stay
unbounded. For our purposes we will show that the set above with $a_2=-1$ and $a_1=a$ is dense in $\mathbb{R}$
for any $a\in\mathbb{R}$ with $a>1$. Actually we shall prove that the set
\[
\left\{\frac{\frac{k}{a}-l}{a^k (-1)^l} : k,l\in\mathbb{N}\right\}
\]
is dense in $\mathbb{R}$ for any $a\in\mathbb{R}$ with $a>1$. From
this it should be obvious that the result above follows since the
image of a dense set in $\mathbb{R}\setminus\{0\}$ under the map
$f(x)=1/x$ is also dense in $\mathbb{R}$.

\begin{lemma} \label{densesequence} The set
\[
\left\{\frac{\frac{k}{a}-l}{a^k (-1)^l} : k,l\in\mathbb{N}\right\}
\]
is dense in $\mathbb{R}$ for any $a>1$.
\end{lemma}
\begin{proof}
Let $x\in\mathbb{R}$ and $\epsilon>0$ be given. We want to find $k,l\in\mathbb{N}$ such that
\[
\left|\frac{\frac{k}{a}-l}{a^k (-1)^l}-x\right|<\epsilon.
\]
There are two cases to consider, namely the cases $x>0$ and $x<0$, and we consider them separately (the case $x=0$ is
trivial since keeping $l$ fixed we can find $k$
big enough which does the job).

\medskip

\textit{Case I ($x>0$):} There exists $k\in\mathbb{N}$ such that
$1/a^k<\epsilon/2$. We will show that there exists a positive
odd integer $l=2s-1$ for some $s\in\mathbb{N}$ for which
\[
\left|\frac{\frac{k}{a}-l}{a^{k}(-1)^l}-x\right|=\left|\frac{2s}{a^k}-\frac{1}{a^k}-\frac{k}{a^{k+1}}-x\right|<\epsilon.
\]
But note that this is true since consecutive terms in the sequence $\{2s/a^{k}:s\in\mathbb{N}\}$ are at distance
$2/a^{k}<\epsilon$ units apart and so, for some $s\in\mathbb{N}$ it holds that
$\frac{2s}{a^k}-\frac{1}{a^k}-\frac{k}{a^{k+1}}\in(x-\epsilon,x+\epsilon)$.

\medskip

\textit{Case II ($x<0$):} There exists $k\in\mathbb{N}$ such that
$1/a^k<\epsilon/2$. We will show that there exists a positive
even integer $l=2s$ for some $s\in\mathbb{N}$ for which
\[
\left|\frac{\frac{k}{a}-l}{a^{k}(-1)^l}-x\right|=\left|\frac{k}{a^{k+1}}-\frac{2s}{a^{k}}-x\right|<\epsilon.
\]
But note that this is true since consecutive terms in the sequence $\{2s/a^{k}:s\in\mathbb{N}\}$ are at distance
$2/a^{k}<\epsilon$ units apart and so, for some $s\in\mathbb{N}$ it holds that
$\frac{k}{a^{k+1}}-\frac{2s}{a^{k}}\in(x-\epsilon,x+\epsilon)$.
\end{proof}

\begin{lemma} \label{unbounded}
Let $x\in\mathbb{R}\setminus\{0\}$, $a>1$ and consider sequences
$\{k_i\}$, $\{l_i\}$ of natural numbers with $k_i, l_i\to +\infty$
such that
\[
\frac{\frac{k_i}{a}-l_i}{a^{k_i}(-1)^{l_i}}\to x.
\]
Then both the numerator and denominator stay unbounded.
\end{lemma}
\begin{proof} This is trivial since the denominator grows unbounded and so it forces the numerator to keep up.
\end{proof}

\begin{remark}
The case where $x=0$ is the only one for which one has the freedom of having the denominator grow unbounded and keep the
numerator bounded. However, if one requires
both numerator and denominator to stay unbounded then the numerator can also be made to grow
unbounded (growing at a slower rate than the denominator).
\end{remark}

Let us now proceed with the construction of a locally hypercyclic pair of upper triangular non-diagonal matrices
on $\mathbb{R}^n$ which is not hypercyclic.

\begin{theorem}\label{realuppertriangular}
Let $n$ be a positive integer with $n\geq 2$ and consider the $n\times n$ matrices
\[
A_j=
\left(
\begin{array}{lllcl}
a_j & 0   & 0   & \ldots & 1\\
0   & a_j & 0   & \ldots & 0\\
0   & 0   & a_j & \ldots & 0\\
\vdots & \vdots & \vdots & \ddots & 0\\
0& 0& 0& \ldots & a_j
\end{array}
\right)
\]
for $j=1,2$ where $a_1>1$ and $a_2=-1$. Then $( A_1, A_2 )$ is a locally hypercyclic pair on $\mathbb{R}^n$
which is not hypercyclic. In particular, we have
\[\{x\in\mathbb{R}^n:J_{(A_1,A_2)}(x)=\mathbb{R}^n\}=
\{(x_1,0,\ldots,0)^t\in\mathbb{R}^n: x_1\in\mathbb{R}\}.\]
\end{theorem}
\begin{proof}
It easily follows that
\[
A_1^k A_2^l=
\left(
\begin{array}{lllcl}
a_1^k a_2^l & 0 & 0       & \ldots & a_1^k a_2^l\left(\frac{k}{a_1}+\frac{l}{a_2}\right)\\
0           & a_1^k a_2^l & 0      & \ldots & 0\\
0           & 0           & a_1^k a_2^l & \ldots & 0\\
\vdots & \vdots & \vdots & \ddots & 0\\
0& 0& 0& \ldots & a_1^k a_2^l
\end{array}
\right).
\]
Let $x\neq 0$. We want to find a sequence
$x_i=(x_{i1},x_{i2},\ldots,x_{in})^t$, $i\in\mathbb{N}$ which
converges to the vector $(x,0,\ldots,0)^t$ and such that for any
vector $w=(w_1,w_2,\ldots,w_n)^t$ there exist strictly increasing
sequences $\{k_i\}$, $\{l_i\}$ of positive integers for which
$A_1^{k_i} A_2^{l_i} x_i \to w$. Without loss of generality we may
assume that $w_n\neq 0$. This is equivalent to having
\[
a_1^{k_i} a_2^{l_i} x_{i1} + a_1^{k_i} a_2^{l_i}\left(\frac{k_i}{a_1}+\frac{l_i}{a_2}\right) x_{in}\to w_1
\]
and
\[
a_1^{k_i} a_2^{l_i} x_{ij}\to w_j
\]
for $j=2,\ldots,n$. By Lemma \ref{densesequence} there exist
sequences $\{k_i\}$ and $\{l_i\}$ of positive integers such that $k_i,l_i\to +\infty$ and
\[\frac{a_1^{k_i} a_2^{l_i}}{\frac{k_i}{a_1}+\frac{l_i}{a_2}}\to - \frac{w_n}{x}.\]
We set
\[
x_{i1}=x-\frac{w_1 x}{w_n\left(\frac{k_i}{a_1}+\frac{l_i}{a_2}\right)}, \quad x_{ij}=-\frac{w_j x}
{w_n\left(\frac{k_i}{a_1}+\frac{l_i}{a_2}\right)}
\]
for $j=2,\ldots,n-1$, and
\[
x_{in}=-\frac{x}{\frac{k_i}{a_1}+\frac{l_i}{a_2}}.
\]
Note that, because of Lemma \ref{unbounded}, $x_{i1}\to x$ and
$x_{ij}\to 0$ for $j=2,\ldots,n$. Substituting into the equations
above we find
\[
a_1^{k_i} a_2^{l_i} x_{i1} +  a_1^{k_i} a_2^{l_i}\left(\frac{k_i}{a_1}+\frac{l_i}{a_2}\right) x_{in} =
a_1^{k_i} a_2^{l_i}\left(-\frac{w_1
x}{w_n\left(\frac{k_i}{a_1}+\frac{l_i}{a_2}\right)}\right)\to w_1
\]
and
\[
a_1^{k_i} a_2^{l_i} x_{ij} = a_1^{k_i} a_2^{l_i}\left(-\frac{w_j x}{w_n\left(\frac{k_i}{a_1}+\frac{l_i}{a_2}\right)}
\right) \to w_j
\]
for $j=2,\ldots, n-1$ as well as
\[
a_1^{k_i} a_2^{l_i} x_{in} = a_1^{k_i} a_2^{l_i}\left(-\frac{x}{\frac{k_i}{a_1}+\frac{l_i}{a_2}}\right)\to w_n.
\]

The pair $(A_1,A_2)$ is not hypercyclic. The reason is that
if it is hypercyclic then there is a vector $y=(y_1,y_2,\ldots,y_n)^t\in\mathbb{R}^n$ such that the set
$\{ A_1^kA_2^ly : k,l\in\mathbb{N}\cup\{0\}\}$
is dense in $\mathbb{R}^n$.
Hence the set of vectors
\[\left\{\left( \begin{array}{c}
a_1^k a_2^l y_1+ a_1^k a_2^l\left(\frac{k}{a_1}+\frac{l}{a_2}\right) y_n \\
a_1^k a_2^l y_2\\
\vdots \\
a_1^k a_2^l y_n
\end{array} \right) : k,l\in\mathbb{N}\cup\{0\}\right\}
\]
is dense in $\mathbb{R}^n$. If $y_n=0$ then it is clear that the last coordinate cannot approximate anything but
$0$. If $y_n\neq 0$ then, since $a_1>1$ and $a_2=-1$ the sequence $\{|a_1^k a_2^l y_n|:
k,l\in\mathbb{N}\cup\{0\}\}=\{|a_1|^k |y_n|:k\in\mathbb{N}\cup\{0\}\}$ is geometric and so cannot be dense in
$\mathbb{R}^+$. It is left to the reader to check that \[\{x\in\mathbb{R}^n:J_{(A_1,A_2)}(x)=\mathbb{R}^n\}=
\{(x_1,0,\ldots,0)^t\in\mathbb{R}^n: x_1\in\mathbb{R}\}.\]
\end{proof}

In what follows we establish an analogue of Theorem \ref{realuppertriangular} in the complex setting.

\begin{lemma}\label{complexdenseseq}
Let $a, \theta$ be real numbers such that $a>1$ and $\theta$ an irrational multiple of $\pi$. Then the set
\[ \left\{ \frac{  \frac{k}{ae^{i\theta }}-l}{a^ke^{ik\theta }(-1)^l}: k,l\in\mathbb{N}\right\} \]
is dense in $\mathbb{C}$.
\end{lemma}
\begin{proof}
Let $w=|w|e^{i\phi} \in \mathbb{C} \setminus \{ 0\}$ and $\epsilon >0$. By the denseness of the irrational rotation on
the unit circle
and by the choice of $a$, there exists a
positive integer $k$ such that
\[ |e^{-ik\theta }-e^{i\phi}|<\frac{\epsilon}{4|w|} \quad \textrm{and} \quad \frac{k}{a^{k-1}}<\frac{\epsilon }{4} .\]
By the proof of Lemma \ref{densesequence} there exists a non-negative odd integer $l=2s-1$ for some $s\in \mathbb{N}$
such that
\[ \left|\frac{-l}{a^k(-1)^l}-|w| \right|=\left|\frac{2s}{a^k}-\frac{1}{a^k}-|w|\right|<\frac{\epsilon }{2}.\]
Using the above estimates it follows that
\begin{eqnarray*}
\left| \frac{  \frac{k}{ae^{i\theta }}-l}{a^ke^{ik\theta }(-1)^l}-|w|e^{i\phi}\right|
&\leq & \left|\frac{  \frac{k}{ae^{i\theta }}}{a^ke^{ik\theta }(-1)^l}\right|+\left|\frac{-l}{a^ke^{ik\theta }
(-1)^l}-|w|e^{i\phi}\right|\\
&\leq&\frac{k}{a^{k-1}}+\left|\frac{-l}{a^k(-1)^l}-|w|\right|+|w||e^{-ik\theta}-e^{i\phi}| \\
&<& \frac{\epsilon}{4}+\frac{\epsilon}{2}+\frac{\epsilon}{4}
=\epsilon.
\end{eqnarray*}
\end{proof}
We now construct a pair of upper-triangular non-diagonal matrices which is locally hypercyclic on $\mathbb{C}^n$
and not hypercyclic.

\begin{theorem}\label{complexuppertriangular}
Let $n$ be a positive integer with $n\geq 2$ and consider the $n\times n$ matrices
\[
A_j=
\left(
\begin{array}{lllcl}
a_j & 0   & 0   & \ldots & 1\\
0   & a_j & 0   & \ldots & 0\\
0   & 0   & a_j & \ldots & 0\\
\vdots & \vdots & \vdots & \ddots & 0\\
0& 0& 0& \ldots & a_j
\end{array}
\right)
\]
for $j=1,2$ where $a_1=a e^{i\theta}$ for $a>1$, $\theta$ an irrational multiple of $\pi$ and $a_2=-1$. Then $(
A_1, A_2 )$ is a locally hypercyclic pair on $\mathbb{C}^n$ which is not hypercyclic. In particular, we have
\[\{z\in\mathbb{C}^n:J_{(A_1,A_2)}(z)=\mathbb{C}^n\}=\{(z_1,0,\ldots,0)^t\in\mathbb{C}^n: z_1\in\mathbb{C}\}.\]
\end{theorem}
\begin{proof}
The proof follows along the same lines as the proof of Theorem \ref{realuppertriangular} where use is
made of Lemma \ref{complexdenseseq} instead of Lemma \ref{densesequence}.
\end{proof}

\begin{remark}
Note that for $n=2$ the upper triangular matrices we obtain in Theorems \ref{realuppertriangular} and
\ref{complexuppertriangular} are in Jordan form. This gives an example of a locally hypercyclic pair of matrices
in Jordan form which is not hypercyclic.
\end{remark}

\section{Concluding remarks and questions}
We stress that all the tuples considered in this work consist of commuting matrices/operators. Recently, in
\cite{Java1} Javaheri deals with the non-commutative case. In particular, he shows that for every positive
integer $n\geq 2$ there exist non-commuting linear maps $A, B:\mathbb{R}^n\to \mathbb{R}^n$  so that for every
vector $x=(x_1, x_2,\ldots , x_n)$ with $x_1\neq 0$ the set $$\{ B^{k_1}A^{l_1}\ldots B^{k_n}A^{l_n}x:
k_j,l_j\in \mathbb{N}\cup \{ 0\} , 1\leq j\leq n\} $$ is dense in $\mathbb{R}^n$. In other words the $2n$-tuple
$(B,A, \ldots ,B,A)$ is hypercyclic.

The following open question was kindly posed by the referee.

\smallskip

\noindent \textbf{Question.} Suppose $(T_1, T_2, \ldots ,T_m)$ is a locally hypercyclic tuple of (commuting)
matrices such that $J_{(T_1, T_2, \ldots ,T_m)}(x)=\mathbb{R}^n$ for a finite set of vectors $x$ in
$\mathbb{R}^n$ whose linear span is equal to $\mathbb{R}^n$. Is it true that the tuple $(T_1, T_2, \ldots ,T_m)$
is hypercyclic? Similarly for $\mathbb{C}^n$.

\smallskip

\noindent \textbf{Acknowledgements.} We would like to thank the referee for an extremely careful reading of the
manuscript. Her/his remarks and comments helped us to improve considerably the presentation of the paper.
Finally we mention that Theorem 4.1 is due to the referee.

\end{document}